\documentclass[12pt]{amsart}  % draft

\usepackage{custom_preamble} % showkeys

\begin{document}

\title[Approximate groups]{Approximate (Abelian) groups}

\author{\tsname}
\address{\tsaddress}
\email{\tsemail}

\begin{abstract}
Our aim is to discuss the structure of subsets of Abelian groups which behave `a bit like' cosets (of subgroups).  One version of `a bit like' can be arrived at by relaxing the usual characterisation of cosets: a subset $S$ of an Abelian group is a coset if for every three elements $x,y,z\in S$ we have $x+y-z \in S$.  What happens if this is not true $100\%$ of the time but is true, say, $1\%$ of the time?  It turns out that this is a situation which comes up quite a lot, and one possible answer is called Fre{\u\i}man's theorem.  We shall discuss it and some recent related quantitative advances.
\end{abstract}

\maketitle

\section{Introduction}

The aim of this article is to cover some of the recent developments in the theory of approximate Abelian groups.  Our starting point is a common characterisation of cosets of subgroups: suppose that $G$ is an Abelian group and $A \subset G$ is a coset of a subgroup of $G$ -- we call this a coset \emph{in} $G$.  A simple characterisation of $A$ being a coset in $G$ is that
\begin{equation*}
\textrm{(i) } A \neq \emptyset; \textrm{ and (ii) } x,y,z \in A \Rightarrow x+y-z \in A.
\end{equation*}
The theory of approximate groups is concerned with relaxing these conditions.  Relaxing the first does not deliver particularly exciting results; relaxing the second, however, turns out to be very fruitful.

Our relaxations will be statistical in nature, and so we shall think of $G$ as being discrete and endowed with Haar counting measure.  It follows that we shall be interested in finite sets $A$.  We write
\begin{equation*}
E(A):=\sum_{x,y,z \in G}{1_A(x)1_A(y)1_A(z)1_A(x+y-z)},
\end{equation*}
a quantity which is called the \emph{additive energy} of $A$.  We see that $A$ is a coset if and only if (it is non-empty) and $E(A)=|A|^3$.  Our first question is what happens if condition (ii) is true only a proportion $1-\delta$ of the time.  In particular, what do sets $A$ look like for which
\begin{equation}\label{eqn.largeut}
E(A) \geq (1-\delta)|A|^3,
\end{equation}
where $\delta$ is to be thought of as a small constant, say $\delta \leq 1/10$, and $|A|$ is to be thought of as tending to infinity.

It is instructive to begin with some examples.  The natural way to create sets with this property is based around cosets in $G$.  Indeed, suppose that $H$ is a coset in $G$ and $A$ is any set satisfying
\begin{equation*}
|A \cap H | \geq (1-\epsilon)|A| \textrm{ and }|A\cap H| \geq (1-\eta)|H|.
\end{equation*}
In words this says that $1-\epsilon$ of $A$ is contained in a coset $H$, and $1-\eta$ of $H$ is in the intersection of $A$ and $H$.  Then, after a short calculation, we find that
\begin{equation*}
E(A) \geq (1-O(\epsilon + \eta))|A|^3.
\end{equation*}
It turns out that sets constructed in the above way are essentially the \emph{only} sets with large additive energy in the sense of (\ref{eqn.largeut}).  The following result is classical and has been considered in far more generality than the statement here suggests.  A proof can be read out of Fournier \cite{fou::}, but it seems likely that it was known before then.
\begin{proposition}\label{prop.firstrelax}
Suppose that $G$ is an Abelian group and $A \subset G$ is finite with $E(A) \geq (1-\delta)|A|^3$.  Then there is some coset $H$ in $G$ such that
\begin{equation*}
|A \cap H | \geq (1-O(\delta^{1/2}))|A| \textrm{ and } |A \cap H | \geq (1-O(\delta^{1/2}))|H|.
\end{equation*}
\end{proposition}
The main strength of this result is that it is a rough equivalence: every set satisfying the conclusion also satisfies the hypothesis with $\delta$ replaced by $O(\delta^{1/2})$ so that up to powers the hypothesis and conclusion are equivalent.

The main weakness of the result is that there may be very few sets satisfying the hypothesis.  For example, suppose that $G=\Z/p\Z$ where $p$ is a prime.  Then $G$ has no non-trivial subgroups, so if $A$ is of `intermediate' size then a short argument from Proposition \ref{prop.firstrelax} tells us we must have $E(A) \leq (1-\Omega(1))|A|^3$; equivalently, no set of `intermediate' size satisfies the hypothesis.

This weakness highlighted in the above discussion can be rectified by a further relaxation of condition (ii), and this is the main concern of the paper.  We ask what happens if condition (ii) is true only a proportion $\delta$ of the time.  In particular, what do sets $A$ look like for which
\begin{equation}\label{eqn.large}
E(A) \geq \delta|A|^3,
\end{equation}
where this time $\delta$ is to be thought of as tending to $0$ (if at all) much more slowly than $|A|$ tends to infinity.

Once again we start by trying to construct examples of such sets.  As before we can use cosets to generate a large class of sets with large additive energy (in the sense of (\ref{eqn.large}) this time), but what is more interesting is that a genuinely new sort of structure emerges, that of arithmetic progressions.

If $P$ is an arithmetic progression a short calculation shows that $E(P) \sim \frac{2}{3}|P|^3$, but it turns out that this is just one example from a much wider class.

\begin{definition}[Convex coset progressions]  A \emph{convex progression} in $G$ is a set of the form $\phi(Q \cap \Z^d)$ where $Q$ is a symmetric convex body about the origin in $\R^d$, and $\phi:\Z^d \rightarrow G$ is a homomorphism.  A \emph{convex coset progression} in $G$ is then a set $H+P$ where $P$ is a convex progression and $H$ is a coset in $G$.  In both cases we say that the progression is \emph{$d$-dimensional}.
\end{definition}
It is worth making two small remarks here.  First, dimension is monotonic so that if a convex coset progression $M$ is $d$-dimensional then it is also $d'$-dimensional for all $d' \geq d$; secondly, in this article we are only interested in dimension up to a constant multiple.

Crucially convex coset progressions inherit growth properties from the convex body used in their definition.  It is not hard to show that if $M$ is a convex coset progression then
\begin{equation*}
|M+M| \leq \exp(O(d))|M|.
\end{equation*}

We now return to constructing sets with large additive energy.  Suppose that $M$ is a $d$-dimensional convex coset progression and $A$ is any set satisfying
\begin{equation*}
|A \cap M| \geq \eta |A|\textrm{ and } |A \cap M| \geq \epsilon |M|.
\end{equation*}
Then a short calculation using the fact that $|M+M| \leq \exp(O(d))|M|$ tells us that
\begin{equation*}
E(A) \geq \epsilon \eta^3\exp(-O(d))|A|^3.
\end{equation*}
Again it turns out that sets constructed in the above way are essentially the only sets having large additive energy.  The following result captures this fact and is a combination of the Balog-Szemer{\'e}di Lemma \cite{balsze::} and the Green-Ruzsa Theorem \cite{greruz::0}.  It should be remarked that the Green-Ruzsa Theorem is also called Fre{\u\i}man's theorem for Abelian groups and extends Fre{\u\i}man's theorem \cite{fre::} from $\Z$ to general Abelian group. 
\begin{theorem}\label{thm.frei}
Suppose that $G$ is an Abelian group and $A \subset G$ is finite with $E(A) \geq \delta|A|^3$.  Then there is a convex coset progression $M$ such that
\begin{equation*}
|A \cap M| \geq \eta(\delta)|A|\textrm{, } |A\cap M| \geq \epsilon(\delta)|M| \textrm{ and } \dim M \leq d(\delta),
\end{equation*}
for some (increasing) functions $\eta,\epsilon$ and $d$.
\end{theorem}
While the result is appealing in its own right, since the breakthrough work of Gowers \cite{gow::4} it has become a central result in additive combinatorics and allied areas as a result of numerous applications.  This wealth of applications provides strong empirical evidence for the utility of the theorem, but there are also some rather compelling theoretical reasons why it should be so useful.  We turn to some of these now.
\begin{enumerate}
\item The hypothesis of the theorem is robust under small perturbations.  This is particularly useful because while the input is flexible, the output, a convex coset progression, is rather rigid.
\item The hypothesis of the theorem is easily satisfied.  From a theoretical perspective this is because convex coset progressions are ubiquitous in contrast to subgroups (in some groups).
\item A convex coset progression supports a lot of structure. While it is not a coset, it behaves enough like a coset that it can support many commonly used analytic arguments, and in particular a sort of approximate harmonic analysis.  This means that many results for groups can also be established for convex coset progressions.
\item The result is a rough equivalence: any set satisfying the conclusion of the theorem satisfies the hypothesis with $\delta$ replaced by $\epsilon(\delta)\eta(\delta)^3\exp(-O(d(\delta)))$.\end{enumerate}
The quality of the rough equivalence serves as a measure of the strength of Theorem \ref{thm.frei} and conjecturally this equivalence is polynomial.  Our interest in this paper is in the quality of this equivalence -- the strength of the bounds on $\eta(\delta),\epsilon(\delta)$ and $d(\delta)$ -- the stronger they are the stronger the results in applications.  Conjecturally we can take
\begin{equation*}
\log \eta(\delta)^{-1}, \log \epsilon(\delta)^{-1},d(\delta) = O(\log \delta^{-1})
\end{equation*}
in Theorem \ref{thm.frei}.  This is called the Polynomial Fre{\u\i}man-Ruzsa conjecture, and if true means that any set satisfying the output of Theorem \ref{thm.frei} automatically satisfies the hypothesis with $\delta$ replaced by $\delta^{O(1)}$ -- the rough equivalence would be rather strong.

Combining Gowers' refinement of the Balog-Szemer{\'e}di Lemma \cite{gow::4} with the Green-Ruzsa Theorem \cite{greruz::0} gives
\begin{equation*}
\log \eta(\delta)^{-1}, \log \epsilon(\delta)^{-1},d(\delta) = O(\delta^{-O(1)}).
\end{equation*}
Green and Ruzsa actually gave an explicit value for the constant implied in the $O(1)$ term and there was some work improving this constant before an important breakthrough by Schoen \cite{sch::1} who showed that it is smaller than any power.  Specifically he proved that one may take
\begin{equation}\label{eqn.sch}
\log \eta(\delta)^{-1}, \log \epsilon(\delta)^{-1},d(\delta) = O(\exp(O(\sqrt{\log \delta^{-1}}))).
\end{equation}
Following on from this we were recently able to show in \cite{san::00} that
\begin{equation}\label{eqn.finalbounds}
\log \eta(\delta)^{-1}, \log \epsilon(\delta)^{-1},d(\delta) = O(\log^{O(1)}\delta^{-1}).
\end{equation}
The results we have chosen to mention above are far from a complete history of work on Theorem \ref{thm.frei}.  Indeed, as a centrepiece of additive combinatorics it has been investigated from many different angles, but we do not have the space to discuss these here.  The interested reader may wish to consult the notes \cite{ruz::03} of Ruzsa.

\section{De-coupling the argument}

The arguments to prove Theorem \ref{thm.frei} separate naturally into two parts: one more combinatorial, and one more algebraic.  The quality of the bounds is almost entirely dependent on the combinatorial part of the argument and that is where most of the recent progress has been made, so we shall now briefly explain how to de-couple the two parts so that we can then focus on the combinatorial one.

The algebraic part of the argument essentially shows that being a convex coset progression is equivalent to satisfying a relative polynomial growth condition.  The latter condition is easier to satisfy combinatorially, so that proving Theorem \ref{thm.frei} comes down to finding a set with relative polynomial growth rather than a convex coset progression.

To be more concrete we start with an observation about convex sets: if $Q$ is a convex set in $\R^d$ then $\mu(nQ) \leq n^d\mu(Q)$ for all $n \geq 1$, and this property is inherited by $d$-dimensional convex coset progressions.  We say that a set $X$ has \emph{relative polynomial growth of order $d$} if
\begin{equation*}
|nX| \leq n^d|X| \textrm{ for all } n \geq 1,
\end{equation*}
where $nX:=X+\dots + X$ and the sum is $n$-fold\footnote{In particular $nX:=\{x_1+\dots+x_n: x_1,\dots,x_n \in X\}$.}.  It turns out that if $M$ is a $d$-dimensional coset progression then $M$ has relative polynomial growth of order $O(d)$, and that having relative polynomial growth is essentially characteristic for convex coset progressions.  To this end we have the following theorem.
\begin{theorem}\label{thm.polygrowi}
Suppose that $G$ is an Abelian group and $X \subset G$ has relative polynomial growth of order $d$.  Then there is a (centred) convex coset progression $M$ in $G$ such that
\begin{equation*}
X-X \subset M\textrm{, } |M| \leq \exp(O(d\log d))|X| \textrm{ and } \dim M =O(d\log d).
\end{equation*}
\end{theorem}
We shall sketch the proof of this Theorem in \S\ref{sec.asymfrei}, but it is not the focus of the paper and is more or less a rearrangement of the ideas in Green and Ruzsa \cite{greruz::0}.

By considering the $N \times \dots \times N$ cube in $\Z^d$ we see that the result is tight up to the logarithmic factors and we think of it as providing an equivalence between $d$-dimensional convex coset progressions and sets with relative polynomial growth of order $d$.  Indeed, instead of proving Theorem \ref{thm.frei} we shall prove the following.
\begin{theorem}\label{thm.combi}
Suppose that $G$ is an Abelian group and $A \subset G$ is such that $E(A) \geq \delta |A|^3$.  Then there is a set $Y$ which is a translate of $X-X$ such that
\begin{equation*}
|A \cap Y| \geq\eta'(\delta) |A|\textrm{, }|A \cap Y| \geq \epsilon'(\delta) |Y|\textrm{ and } |nX| \leq n^{d'(\delta)}|X| \textrm{ for all }n \geq 1
\end{equation*}
for some (increasing) functions $\eta',\epsilon'$ and $d'$.
\end{theorem}
We should like to combine Theorems \ref{thm.polygrowi} and \ref{thm.combi} to get Theorem \ref{thm.frei}.  We can not do this directly but it turns out that by delving a little into the proofs of each one can combine them to do so, and in particular the ways we establish Theorem \ref{thm.combi} for given functions $\eta',\epsilon'$ and $d'$ lead to arguments for establishing Theorem \ref{thm.frei} with $\eta\approx\eta',\epsilon\approx\epsilon'$ and $d\approx d'$.

\section{Overview of the combinatorial obstacles}\label{sec.ov}

Our goal now is to prove Theorem \ref{thm.combi}, and in light of the equivalence mentioned in the previous section we shall think of cosets, convex coset progressions and sets with relative polynomial growth as being the same thing for the purpose of constructing examples.

One of the reasons that proving Theorem \ref{thm.combi} (and hence Theorem \ref{thm.frei}) is hard (and also one reason it is so powerful) is that there are qualitatively three different sorts of structure having large additive energy.  We got a sense of roughly what these are earlier but it is helpful now to record them a little more precisely.
\begin{enumerate}
\item\label{it.r} \emph{(Random sets)} Suppose that $H$ is a coset in $G$ and $A \subset H$ is chosen by including each $h \in H$ independently with probability $\delta$.  Then (with high probability) $E(A)\approx \delta |A|^3$.
\item\label{it.l} \emph{(Independent copies of the same coset)} Suppose $H$ is a coset in $G$ and $A$ is a union of $k\sim \delta^{-1}$ independent cosets in $G/H$.  To be clear this means that $A=\bigcup_{i=1}^k{(x_i+H)}$ where $\{x_i+H\}_{i=1}^k$ is a set of $k$ elements of $H$ such that
\begin{equation*}
n \in \Z^k \textrm{ and } n_1x_1+\dots+n_kx_k\in H \Rightarrow n_ix_i \in H \textrm{ for all } i \in \{1,\dots,k\}.
\end{equation*}
Then $E(A)\approx \delta |A|^3$. 
\item \emph{(Independent copies of different cosets)} Suppose that $k \sim \delta^{-1/2}$ and $H_1$, \dots, $H_k$ are `internally independent cosets' all of the same size which is to be taken to mean that
\begin{equation*}
|H_1+\dots+H_k|=|H_1|\dots |H_k|,
\end{equation*}
and intuitively means that there are no non-trivial relations between elements in the $H_i$s.  Then $E(A) \approx \delta |A|^3$.
\end{enumerate}
It is a little easier to unify the first two classes of example with each other than the third with either of the first two.  This is because in the first two classes there is an obvious choice of coset: $H$.  (In fact the obvious choice is really of subgroup, but this turns out not to be an important distinction here.)  On the other hand in the third class any of the cosets (or, rather, corresponding subgroups of) $H_1$, \dots, $H_k$ are reasonable choices and there is no particular reason to pick one over the other.

A unifying aspect of the first two classes above is that the sets $A$ given have small sumset.  In particular, in both classes we have that $|A+A| =\Theta(\delta^{-1}|A|)$ which we think of as saying that $A$ has `small doubling'; in the third class $|A+A| = \Omega(|A|^2)$ so that it is almost as large as can be.

The first step in proving Theorem \ref{thm.combi} is then in converting sets from the third class into the first two, and this is the purpose of the Balog-Szemer{\'e}di-Gowers Lemma.  Qualitatively this was proved by Balog and Szemer{\'e}di in \cite{balsze::}, but it was a very important step when Gowers established polynomial bounds in \cite{gow::4}.
\begin{theorem}[Balog-Szemer{\'e}di-Gowers Lemma]\label{thm.bsg}
Suppose that $A \subset G$ has $E(A) \geq \delta|A|^3$.  Then there is a set $A' \subset A$ such that
\begin{equation*}
|A'| \geq \delta^{O(1)}|A|  \textrm{ and } |A'+A'| \leq \delta^{-O(1)}|A'|.
\end{equation*}
\end{theorem}
This result has been studied extensively elsewhere and will not be our focus here.  The interested reader might like to consult the book \cite{taovu::} of Tao and Vu.  It is work saying that the proof is elementary, albeit rather clever, and is set around the idea of examining $A \cap (x+A)$ for suitable randomly chosen $x$.  In the third class of examples considered above this has the effect of selecting one of the cosets (at random).

From now on we shall be interested in the case of so-called `small doubling' mentioned earlier meaning the case when $|A+A| \leq K|A|$, and shall prove the following.
\begin{theorem}\label{thm.combi2}
Suppose that $G$ is an Abelian group and $A \subset G$ is such that $|A+A| \leq K|A|$.  Then there is a set $Y$ which is a translate of $X-X$ such that
\begin{equation*}
|A \cap Y| \geq\eta''(K) |A|\textrm{, }|A \cap Y| \geq \epsilon''(K) |Y|\textrm{ and } |nX| \leq n^{d''(K)}|X| \textrm{ for all }n \geq 1
\end{equation*}
for some (decreasing) functions $\eta'',\epsilon''$ and $d''$.
\end{theorem}
This yields Theorem \ref{thm.combi} on combination with the Balog-Szemer{\'e}di-Gowers Lemma with
\begin{equation*}
\eta'(\delta) = \delta^{O(1)}\eta''(\delta^{-O(1)}), \epsilon'(\delta) = \delta^{O(1)}\epsilon''(\delta^{-{O(1)}}) \textrm{ and } d'(\delta) = d''(\delta^{-O(1)}).
\end{equation*}
In actual fact for the best bounds one also goes into the details of the proof of the Balog-Szemer{\'e}di-Gowers Lemma, but for this overview that improvement will not concern us.

The focus of the paper now is on proving Theorem \ref{thm.combi2} and we split into three sections.  In \S\ref{sec.basic} we establish Theorem \ref{thm.combi} with bounds corresponding roughly to the original work of Green and Ruzsa; in \S\ref{sec.schoen} we develop Schoen's improvement of this; and, finally, in \S\ref{sec.new} we develop the improvement leading to the bounds in (\ref{eqn.finalbounds}).

\section{Sumset estimates and polynomial growth}

In this brief section we record a couple of useful results which will help us to establish relative polynomial growth on all scales from a growth condition on just one scale.  In particular we have the following result of Chang \cite{cha::0}.
\begin{lemma}[Variant of Chang's covering lemma]\label{lem.ccl}
Suppose that $G$ is an Abelian group and $X \subset G$ is symmetric with $|(3k+1)X| < 2^k|X|$ for some $k$. Then $|nX| \leq n^k|X|$ for all $n \geq 1$.
\end{lemma}
To establish the hypothesis of this lemma it will also be useful to have Pl{\"u}nnecke's inequality.
\begin{theorem}[Pl{\"u}nnecke's inequality]  Suppose that $|A+A| \leq K|A|$.  Then $|nA| \leq K^n|A|$ for all $n \geq 1$.
\end{theorem}
This result was proved by Pl{\"u}nnecke in \cite{plu::}, and his proof was rediscovered and popularised somewhat later by Ruzsa.  Very recently, however, Petridis \cite{pet::} found a new proof which is very direct and well worth reading.

Before closing this short section it is worth saying that the above two results are part of a rich family of sumset estimates we do not have time to touch on here, but we direct the interested reader towards Tao and Vu \cite{taovu::} for more details.

\section{The basic argument}\label{sec.basic}

We now turn our attention to proving Theorem \ref{thm.combi2} with bounds of the quality arrived at by Green and Ruzsa in \cite{greruz::0}.  We are thus considering a set $A$ with $|A+A| \leq K|A|$ and in light of our earlier discussions we can restrict ourselves to sets coming from the first two classes of structure in \S\ref{sec.ov}.  

Even with the work we have done the two classes of possible structure behave differently: in the first class when $A$ is chosen randomly from $H$, $A$ will typically have a lot of gaps so that $1_A$ is not very smooth.  One way of smoothing a function is by averaging or convolving and to this end we make a definition.

Given $f,g \in \ell^1(G)$ we define the \emph{convolution} of $f$ and $g$ to be the function
\begin{equation*}
f \ast g(x) = \sum_{y+z=x}{f(y)g(z)} \textrm{ for all }x \in G.
\end{equation*}
To relate this to sets we define some level sets called symmetry sets.  Given a set $A$ the \emph{symmetry set at threshold $\eta$} is defined to be the set
\begin{equation*}
\Sym_{\eta}(A):=\{x \in G:1_A \ast 1_{-A}(x) \geq \eta |A|\}.
\end{equation*}
Note that $1_A \ast 1_{-A}(x) \leq |A|$ so that $\Sym_\eta(A)$ is the set of points where $1_A \ast 1_{-A}$ is a proportion $\eta$ of its maximum.

Heuristically we expect $1_A \ast 1_A$ to be pretty smooth -- on a qualitative level $1_A$ is an element of $L^2$ and the convolution of two $L^2$ functions is continuous.  Concretely then we expect to have a quantitative version of the notion of uniform continuity, meaning there should be a (large) set $X$ such that on translation by elements of $X$, the convolution does not vary very much.  To formulate this precisely we define translation: given $f \in \ell^2(G)$ we write
\begin{equation*}
\tau_x(f)(y) = f(y+x) \textrm{ for all }x,y \in G.
\end{equation*}
It will also be helpful to write $\mu_A$ for the function $1_A/|A|$ so that $f \ast \mu_A(x)$ is the average value of $f$ over the set $x-A$.  The heuristic above can be made precise in a number of ways but one very powerful approach is due to Croot and Sisask \cite{crosis::} who proved the following lemma.
\begin{lemma}[Croot-Sisask Lemma]  Suppose that $G$ is an Abelian group, $f \in \ell^2(G)$ and $|A+A| \leq K|A|$.  Then there is a set $X$ with $|X| \geq (2K)^{-O(\epsilon^{-2})}|A|$ such that
\begin{equation*}
\|\tau_x(f \ast \mu_A) - f \ast \mu_A\|_{\ell^2(G)} \leq \epsilon \|f\|_{\ell^2(G)} \textrm{ for all } x \in X.
\end{equation*}
\end{lemma}
\begin{proof}[Sketch proof]
  The basic idea is to start with the equality
\begin{equation*}
f \ast \mu_A = \E_{a \in A}{\tau_{-a}(f)}.
\end{equation*}
Given this we can randomly sample from $A$, say $k$ times, and get a good approximation to $f \ast \mu_A$:
\begin{equation*}
f \ast \mu_A \approx \frac{1}{k}\sum_{i=1}^k{\tau_{-z_i}(f)}
\end{equation*}
for $z_1,\dots,z_k$ chosen uniformly at random from $A$.  Here `$\approx$' means approximately equal in $\ell^2$-norm, and the larger the value of $k$, the better the approximation.

We now examine the set $L$ of vectors $(z_i)_{i=1}^k$ such that the approximation is good.  By averaging we prove that $L-L$ has a large intersection, call it $X$, with the diagonal set $\{(a,\dots,a):a \in A\}$.  On the other hand, if $x \in X$ then it follows that there is some $z \in L$ such that
\begin{equation*}
f \ast \mu_A \approx \frac{1}{k}\sum_{i=1}^k{\tau_{-z_i-x}(f)} \textrm{ and } f \ast \mu_A \approx \frac{1}{k}\sum_{i=1}^k{\tau_{-z_i}(f)}
\end{equation*}
and hence
\begin{equation*}
\tau_x(f \ast \mu_A) \approx f \ast \mu_A \textrm{ for all } x \in X.
\end{equation*}
Working through the details of this sketch gives the proof.
\end{proof}

Given the above result we shall now prove the following which is the version of Theorem \ref{thm.combi2} corresponding to the bounds of Green and Ruzsa although the argument is somewhat different.
\begin{theorem}\label{thm.combiweak}
Suppose that $G$ is an Abelian group and $A \subset G$ is such that $|A+A| \leq K|A|$.  Then there is a set $Y$ which is a translate of $X-X$ such that
\begin{equation*}
|A \cap Y| \geq \exp(-K^{1+o(1)}) |A| \textrm{, }|A \cap Y| = \Omega(|Y|/K)
\end{equation*}
and
\begin{equation*}
 |nX| \leq n^{K^{1+o(1)}}|X| \textrm{ for all }n \geq 1.
\end{equation*}
\end{theorem}
\begin{proof}[Sketch proof]
We let $k$ be a natural number to be optimised later and apply the Croot-Sisask lemma with $\epsilon = 1/2k\sqrt{K}$ to get a set $X$ with $|X| \geq |A|/(2K)^{O(k^2K)}$ such that
\begin{equation*}
\|\tau_x(1_A \ast 1_A) - 1_A \ast 1_A\|_{\ell^2(G)}^2 \leq |A|^3/4K \textrm{ for all } x \in kX
\end{equation*}
by the triangle inequality.  Now, by Cauchy-Schwarz we have $\|1_A \ast 1_A\|_{\ell^2(G)}^2 \geq |A|^3/K$ which by the triangle inequality and the output of the Croot-Sisask lemma tells us two things:
\begin{equation*}
kX \subset 2A-2A \textrm{ and }\|1_A \ast 1_A\ast \mu_{X-X} \|_{\ell^2(G)}^2 \geq |A|^3/4K.
\end{equation*}
The second of these gives us the translate $Y$ of $X-X$ such that $|A \cap Y| = \Omega(|Y|/K)$ via an averaging argument; the first let us control the degree of polynomial growth of $X$.

Since $kX \subset 2A-2A$ we have by Pl{\"u}nnecke's inequality that
\begin{equation*}
|klX| \leq K^{4l}|A| \leq K^{4l}(2K)^{O(k^2K)}|X|.
\end{equation*}
Putting $3r+1=kl$ and $l=k^2K$ we get that
\begin{equation*}
|(3r+1)X| \leq K^{O(r^{2/3}K^{1/3})}|X|;
\end{equation*}
it follows that we can take $r=O(K\log^3K)$ such that $|(3r+1)X| < 2^r|X|$.  Chang's covering lemma then tell us that $X$ has the right order of relative polynomial growth.

Finally from the definition of $r$ and $l$ in terms of $k$ we get that $k=O(\log K)$ from which the bound in the size of $|A\cap Y|/|A|$ follows.
\end{proof}

\section{Schoen's refinement}\label{sec.schoen}

Schoen in \cite{sch::1} made a major breakthrough when he proved the bounds mentioned in (\ref{eqn.sch}).  If we study the argument above the weakness was that we had to take $\epsilon \approx 1/\sqrt{K}$ in our application of the Croot-Sisask lemma, and since the resultant set $X$ has size exponentially dependent on $\epsilon^{-2}$ this lead to exponential losses in $K$.

To some extent these loses are necessary as can be seen by considering the examples in class (\ref{it.r}) of \S\ref{sec.ov}.  In this class $A$ is chosen randomly with probability $1/K$ from a coset $H$, so that (with high probability)
\begin{equation*}
1_A \ast 1_A(x) \approx |A|/K \textrm{ for all }x \in A-A.
\end{equation*}
On the other hand $A+A$ is very structured -- it is the whole coset $H$.  Similarly, in class (\ref{it.l}) of \S\ref{sec.ov}, $A+A$ is again very structured.

It follows that in either of the above cases $1_{A+A} \ast 1_{A+A}$ takes rather large values; certainly much larger than those of $1_A \ast 1_A$ in the case when $A$ is chosen randomly.  If we can guarantee that some convolution takes a lot of values much larger than its average value then the arguments at the end of the last section can be applied much more effectively. 

This is roughly speaking Schoen's idea and the following is one of the key ingredients from \cite{sch::1}.
\begin{proposition}\label{prop.intit}
Suppose that $G$ is an Abelian group, $A$ is a finite subset of $G$ with $|A+A| \leq K|A|$, and $\epsilon \in (0,1]$ is a parameter. Then there is a non-empty set $A'\subset A$ such that
\begin{equation*}
|\Sym_{K^{-\eta}}(A'+A)| \geq \exp(-\exp(O(\eta^{-1}))\log K)|A|.
\end{equation*}
\end{proposition}
The proof of this is iterative and based around an important observation which seems to have been first made by Katz and Koester in \cite{katkoe::}.

Suppose that $A'' \subset G$ is such that $|A+A''| \leq M|A|$ and $|A'' + A''| \leq L|A''|$.  Then we have
\begin{equation*}
1_{A+A''} \ast 1_{-(A+A'')}(x) = |(A+A'') \cap (x+A+A'')| \geq |A + (A'' \cap (x+A''))|.
\end{equation*}
Writing $S$ for the set of $x$ such that $A'' \cap (x+A'')$ is large, that is
\begin{equation*}
S:=\{x \in G: 1_{A''} \ast 1_{A''}(x) \geq |A''|/2L\},
\end{equation*}
we have two possibilities:
\begin{enumerate}
\item either $1_{A+A''} \ast 1_{-(A+A'')}(x) \geq R|A+A''|$ for all $x \in S$;
\item or, putting $A''':=A''\cap (x+A'')$, we have
\begin{equation*}
|A'''+A'''| \leq 2L^2|A'''|\textrm{ and }|A+A'''| \leq (M/R)|A|.
\end{equation*}
\end{enumerate}
Given this we proceed by downward induction on $|A+A''|/|A|$ terminating when we are in the first case and repeating with $A''$ replaced by $A'''$, $M$ by $M/R$ and $L$ by $2L^2$ in the second.  This yields the proposition.

With the above result we can now prove the following.
\begin{theorem}\label{thm.combisch}
Suppose that $G$ is an Abelian group and $A \subset G$ is such that $|A+A| \leq K|A|$.  Then there is a set $Y$ which is a translate of $X-X$ such that
\begin{equation*}
|A \cap Y| \geq \exp(-\exp(O(\sqrt{\log K}))) |A| \textrm{, }|A \cap Y| =\Omega(|Y|/K^{O(1)})
\end{equation*}
and
\begin{equation*}
 |nX| \leq n^{\exp(O(\sqrt{\log K})}|X| \textrm{ for all }n \geq 1.
\end{equation*}
\end{theorem}
\begin{proof}[Sketch proof]
We apply Proposition \ref{prop.intit} and put $S:= \Sym_{K^{-\eta}}(A'+A)$ so that
\begin{equation*}
|S| \geq  \exp(-\exp(O(\eta^{-1}))\log K)|A|.
\end{equation*}
By definition and the Cauchy-Schwarz inequality we have that
\begin{equation*}
\|1_{A+A'} \ast 1_S\|_{\ell^2(G)} \geq K^{-2\eta}|A+A'||S|^2,
\end{equation*}
and (since $S \subset 2A-2A$) that
\begin{equation*}
|S+S| \leq \exp(\exp(O(\eta^{-1}))\log K)|S|.
\end{equation*}
We then proceed as in the proof of Theorem \ref{thm.combiweak} but this time apply Croot-Sisask to the function $1_{A+A'}$ and the set $S$ with parameter $\epsilon = 1/2kK^{-\eta}$ and get a set $X$ with
\begin{equation*}
|(3r+1)X| \leq (2K)^{O(r^{2/3}K^{O(\eta)}\exp(O(\eta^{-1})))}|X|.
\end{equation*}
Optimising for $\eta$ we take $\eta = 1/\sqrt{\log K}$, and then the argument proceeds much as before to give the result.
\end{proof}

\section{The L{\'o}pez-Ross trick and generalised Croot-Sisask}\label{sec.new}

The Croot-Sisask lemma has a rather powerful generalisation to $\ell^p$-norms.
\begin{lemma}[Croot-Sisask lemma, $\ell^p$-norm version]  Suppose that $G$ is an Abelian group, $f \in \ell^p(G)$ and $|A+A| \leq K|A|$.  Then there is a set $X$ with $|X| \geq (2K)^{-O(\epsilon^{-2}p)}|A|$ such that
\begin{equation*}
\|\tau_x(f \ast \mu_A) - f \ast \mu_A\|_{\ell^p(G)} \leq \epsilon \|f\|_{\ell^p(G)} \textrm{ for all } x \in X.
\end{equation*}
\end{lemma}
This result is also from \cite{crosis::} and the proof of the $\ell^2$ version except that Khintchine's inequality has to be replaced by the Marcinkiewicz-Zygmund inequality.  

The reason that this result is so much more powerful than the $\ell^2$ version of the Croot-Sisask lemma is in the bounds.  In particular the $p$ dependence is exponential in $p$, rather than doubly exponential which is what all previous arguments had given.  To understand why it is useful here we now sketch the proof of the following.
\begin{theorem}\label{thm.combifinal}
Suppose that $G$ is an Abelian group and $A \subset G$ is such that $|A+A| \leq K|A|$.  Then there is a set $Y$ which is a translate of $X-X$ such that
\begin{equation*}
|A \cap Y| \geq \exp(-\log^{O(1)}K) |A|\textrm{, }|A \cap Y| =\Omega(|Y|/K^{O(1)})
\end{equation*}
and
\begin{equation*}
 |nX| \leq n^{\log^{O(1)}K}|X| \textrm{ for all }n \geq 1.
\end{equation*}
\end{theorem}
\begin{proof}[Sketch proof]
Rather than examining the $\ell^2$-norm of the convolution of two functions we use an observation of L{\'o}pez and Ross from \cite{lopros::}:
\begin{equation*}
\langle 1_{A+A}, 1_A \ast 1_A \rangle = |A|^2.
\end{equation*}
On the other hand if we know that
\begin{equation*}
\|\tau_x(1_{A+A} \ast 1_A) - 1_{A+A} \ast 1_A\|_{\ell^p(G)} \leq \epsilon \|1_{A+A}\|_{\ell^p(G)},
\end{equation*}
then we conclude that
\begin{equation*}
\langle \tau_x(1_{A+A}), 1_A \ast 1_A \rangle \geq  |A|^2 - \epsilon \|1_{A+A}\|_{\ell^p(G)}|A| = |A|^2(1-\epsilon K^{1/p}).
\end{equation*}
We conclude that we can take $p \sim \log K$ and $\epsilon = \Omega(1)$ such that
\begin{equation*}
\langle \tau_x(1_{A+A}), 1_A \ast 1_A \rangle \geq  |A|^2/2.
\end{equation*}
But this means by the Croot-Sisask lemma that there is a set $X$ of size at least $|A|K^{-O(k^2)}$ such that
\begin{equation*}
\langle \tau_x(1_{A+A}), 1_A \ast 1_A \rangle \geq  |A|^2/2 \textrm{ for all }x \in kX.
\end{equation*}
This can then be plugged back into a similar argument to the ones we had before to get Theorem \ref{thm.combifinal}.

The advantage here is that the set $X$ we have found is a lot bigger than those we had previously found as a result of the good bounds in the Croot-Sisask lemma.
\end{proof}

\section{Polynomial growth and convex progressions}\label{sec.asymfrei}

In this section we shall sketch a proof of Theorem \ref{thm.polygrowi} which we restate now as a reminder.
\begin{theorem}[{Theorem \ref{thm.polygrowi}}]
Suppose that $G$ is an Abelian group and $X \subset G$ is such that $ |nX| \leq n^d|X|$ for all $n \geq 1$.  Then there is a (centred) convex coset progression $M$ in $G$ such that
\begin{equation*}
X-X \subset M\textrm{, } |M| \leq \exp(O(d\log d))|X| \textrm{ and } \dim M =O(d\log d).
\end{equation*}
\end{theorem}
As indicated this is largely a rearrangement of the ideas of Green and Ruzsa in \cite{greruz::0}, which are themselves developed from the hugely influential paper \cite{ruz::9} of Ruzsa.

One of the key tools is the Fourier transform which in this case we define regarding $G$ as a discrete group.  We write $\widehat{G}$ for the compact Abelian group of characters on $G$ and given $f \in \ell^1(G)$ the \emph{Fourier transform} of $f$ is defined to be the function
\begin{equation*}
\widehat{f}:\widehat{G} \rightarrow \C; \gamma\mapsto \sum_{x \in G}{f(x)\overline{\gamma(x)}}.
\end{equation*}
There is a useful notion of approximate annihilator on $G$ called Bohr sets.  Given $\Gamma \subset \widehat{G}$ a compact set and $\delta \in (0,2]$ we write
\begin{equation*}
\Bohr(\Gamma,\delta):=\{x \in G: |\gamma(x)-1| \leq \delta \textrm{ for all } \gamma \in \Gamma\}.
\end{equation*}
Bohr sets interact particularly well with the large spectrum of a set.  Given $A \subset G$ we write
\begin{equation*}
\LSpec(A,\epsilon):=\{\gamma \in \widehat{G}: \|1 - \gamma\|_{L^2(\mu_A \ast \mu_{-A})} \leq \epsilon\}.
\end{equation*}
The basic idea is to show that if $X$ has polynomial growth then $X-X$ is contained in the Bohr set of the large spectrum of (a dilate of) $X$.  We then show that this Bohr set is not too large, and finally that it is actually a low dimensional convex coset progression.

The following proposition deals with the first objective above; it is only slightly more general than \cite[Proposition 4.39]{taovu::}. 
\begin{proposition}\label{prop.lowerbound}
Suppose that $X\subset G$, $l$ is a positive integer such that $|lX| \leq K|(l-1)X|$ and
$\epsilon \in (0,1]$ is a parameter. Then
\begin{equation*}
X-X \subset \Bohr(\LSpec(lX,\epsilon),2\epsilon\sqrt{2K}).
\end{equation*}
\end{proposition}
The proof is fairly straightforward after unpacking the definitions.

The second objective above -- that the Bohr set not be too large -- is proved using an idea of Schoen \cite{sch::0} introduced to Fre{\u\i}man-type problems by Green and Ruzsa in \cite{greruz::0}.
\begin{proposition}\label{prop.LSpecproperties}
Suppose that $X\subset G$ has $|nX| \leq n^d|X|$ for all $n\geq 1$, and $\epsilon \in (0,1/2]$ is a parameter. Then we have the estimate
\begin{equation*}
|\Bohr(\LSpec(X,\epsilon),1/2\pi)| \leq \exp(O(d\log \epsilon^{-1}d))|X|.
\end{equation*}
\end{proposition}
The proof of this is via the Fourier transform which shows that the large spectrum of the specified Bohr set must support a lot of the $\ell^2$-mass of $1_X$.

To deal with similar concerns to those of our the final objective Ruzsa introduced the geometry of number to Fre{\u\i}man-type theorems in \cite{ruz::9}.  There is a great deal to say about this,  and we direct the reader to \cite[Chapter 3.5]{taovu::} for a much more comprehensive discussion.  For our purposes we have the following proposition.
\begin{proposition}\label{prop.coset}
Suppose that $G$ is an Abelian group, $d \in \N$ and $B$ is a Bohr set such that
\begin{equation*}
|\Bohr(\Gamma,(3d+1)\delta)|<2^d|\Bohr(\Gamma,\delta)| \textrm{ for some } \delta < 1/4(3d+1).
\end{equation*}
Then $\Bohr(\Gamma,\delta)$ is a $d$-dimensional convex coset progression.
\end{proposition}
The proof of this involves the covering lemma of Chang mentioned earlier and a very important embedding defined by Ruzsa
\begin{align*}
R_\Gamma:G& \rightarrow  C(\Gamma,\R)\\ x & \mapsto   R_\Gamma(x):\Gamma \rightarrow \R; \gamma \mapsto \frac{1}{2\pi}\arg(\gamma(x)),
\end{align*}
where the argument is taken to lie in $(-\pi,\pi]$.  The map $R_\Gamma$ acts as something called a Fre{\u\i}man morphism\footnote{We direct the unfamiliar reader to \cite[Chapter 5.3]{taovu::}.} which lets us embed Bohr sets into a lattice.

With those three ingredients it is possible to stitch together a proof of Theorem \ref{thm.polygrowi} and the section is complete.

\bibliographystyle{halpha}

\bibliography{references}

\end{document}